\documentclass[a4paper,11pt]{amsart} 
\usepackage{amssymb, amsmath}
\usepackage{amscd}
\numberwithin{equation}{section}
\usepackage{epsfig}
\usepackage{amsmath}
\usepackage{amsfonts,amssymb,amsopn}

\usepackage{amsthm}
\usepackage{hyperref}
\usepackage{verbatim}
\usepackage{color}
\usepackage[color,all]{xy}

\newcommand{\K}{\mathbb{K}}
\newcommand{\PP}{\mathbb{P}}
\newcommand{\TT}{\mathbb{T}}

\newcommand{\cF}{{\mathcal F}}
\newcommand{\cI}{{\mathcal I}}
\newcommand{\cL}{{\mathcal L}}
\newcommand{\cO}{{\mathcal O}}
\newcommand{\cS}{{\mathcal S}}
\newcommand{\cT}{{\mathcal T}}
\newcommand{\cU}{{\mathcal U}}
\newcommand{\cZ}{{\mathcal Z}}

\newcommand{\tV}{\tilde{V}}
\newcommand{\tE}{\tilde{E}}
\newcommand{\tomega}{{\widetilde{\omega}}}
\newcommand{\tpi}{{\widetilde{\pi}}}

\newcommand{\Sym}{\mathrm{Sym}}
\newcommand{\End}{\mathrm{End}\,}
\newcommand{\Ker}{\mathrm{Ker}}
\newcommand{\Image}{\mathrm{Im}\,}
\newcommand{\Iden}{\mathrm{Id}}
\newcommand{\Hom}{\mathrm{Hom}}
\newcommand{\rank}{\mathrm{rk}\,}
\newcommand{\Pic}{\mathrm{Pic}}
\newcommand{\ev}{\mathrm{ev}}

\newcommand{\Kx}{K_{X}}
\newcommand{\Ox}{{{\cO}_{X}}}
\newcommand{\sRat}{\underline{\mathrm{Rat}}\,}
\newcommand{\sPrin}{\mathrm{\underline{Prin}}\,}

\newcommand{\Span}{\mathrm{Span}}
\newcommand{\Gr}{\mathrm{Gr}}
\newcommand{\Quot}{\mathrm{Quot}}
\newcommand{\Hilb}{\mathrm{Hilb}}
\newcommand{\Supp}{\mathrm{Supp}}
\newcommand{\Sec}{\mathrm{Sec}}
\newcommand{\isom}{\xrightarrow{\sim}}

\newcommand{\length}{\mathrm{length}\,}

\newcommand{\Opv}{\cO_{\PP V}}
\newcommand{\Oz}{\cO_{Z}}
\newcommand{\Opvo}{\Opv(1)}
\newcommand{\Opfo}{\cO_{\PP F}(1)}

\newcommand{\defe}{\mathrm{def}}
\newcommand{\tcup}{\tilde{\cup}}
\newcommand{\urd}{U_X (r, d)}
\newcommand{\ure}{U_X (r, e)}

\newtheorem{theorem}{{\textbf Theorem}}[section]
\newtheorem{proposition}[theorem]{{\textbf Proposition}}
\newtheorem{corollary}[theorem]{{\textbf Corollary}}
\newtheorem{lemma}[theorem]{{\textbf Lemma}}
\newtheorem{criterion}[theorem]{{\textbf Criterion}}
\newtheorem{defn}[theorem]{{\textbf Definition}}

\newtheorem{remit}[theorem]{{\textbf Remark}}
\newenvironment{remark}{\begin{remit}\rm}{\end{remit}}
\newenvironment{definition}{\begin{defn}\rm}{\end{defn}}

\newcommand{\fix}[1]{{{\footnotesize{\textcolor{red}{Fix: #1}}}}}

\title{A Riemann--Kempf singularity theorem for higher rank Brill--Noether loci}

\author{George H. Hitching}

\email{george.hitching@hioa.no}

\address{H\o gskolen i Oslo og Akershus, Postboks 4, St. Olavs plass, 0130 Oslo, Norway.}

\begin{document}

\begin{abstract} Given a vector bundle $V$ over a curve $X$, we define and study a surjective rational map $\Hilb^d (\PP V ) \dashrightarrow \Quot^{0, d} ( V^* )$ generalising the natural map $\Sym^d X \to \Quot^{0, d} (\Ox)$. We then give a generalisation of the geometric Riemann--Roch theorem to vector bundles of higher rank over $X$. We use this to give a geometric description of the tangent cone to the Brill--Noether locus $B^r_{r, d}$ at a suitable bundle $E$ with $h^0 (E) = r+k$. This gives a generalisation of the Riemann--Kempf singularity theorem. As a corollary, we show that the $k$th secant variety of the rank one locus of $\PP \End E$ is contained in the tangent cone. \end{abstract}

\maketitle

\section{Introduction}

Let $X$ be a projective smooth curve of genus $g \ge 2$ and $D$ an effective divisor of degree $d$ on $X$. By the geometric Riemann--Roch theorem, $\dim |D|$ is exactly the defect of $D$ on the canonical model of $X$ in $| \Kx |^*$. If $X$ is general, the line bundle $\Ox (D)$ defines a point of multiplicity $h^0 ( X, \Ox (D))$ of the Brill--Noether locus $W_d (X)$. The Riemann--Kempf singularity theorem (see \cite[Chapter 2]{GrHa}) states that the tangent cone to $W_d (X)$ at $\Ox (D)$ is precisely the union of the secants $\Span (D') \subset |\Kx|^*$ for $D' \in |D|$. Using this picture, in \cite{KS}, \cite{CS} and \cite{HM} new proofs of Torelli's theorem were given using the infinitesimal geometry of Brill--Noether loci.

In recent years, higher-rank analogues of $W_d (X)$ have been the subject of much attention. We denote by $\urd$ the moduli space of stable vector bundles of rank $r$ and degree $d$ and consider the \textsl{higher-rank Brill--Noether locus}
\[ B^k_{r,d} \ = \ \{ E \in \urd: h^0 ( X, E) \ge k \} . \]
See \cite{GT} for a summary of relevant results. Our interest is primarily in the infinitesimal geometry of $B^k_{r,d}$ at singular points. As motivation, we note that in \cite{Pau} and \cite{HH} the infinitesimal geometry of generalised theta divisors associated to higher rank vector bundles (examples of \emph{twisted Brill--Noether loci}) was used to prove ``Torelli-type'' theorems (recovering the curve and the bundle respectively). It seems therefore natural to investigate what can be recovered from the tangent cones $\TT_E B^k_{r,d}$ at singular points $E$. Note that \cite{CT} gives a comprehensive introduction and many interesting results on the singular loci of higher rank Brill--Noether loci and twisted Brill--Noether loci.

The projectivised tangent space of $\urd$ at $E$ is $\PP H^1 ( X, \End E )$. It was shown in \cite{HR} that there is a natural map $\psi \colon \PP \End E \dashrightarrow \PP H^1 ( X, \End E )$, generalising the canonical curve, which is an embedding for general $E$. In light of this, as a first step towards finding analogues of the above results on line bundles, one can seek generalisations of the geometric Riemann--Roch and Riemann--Kempf theorems for bundles of higher rank, given in terms of the geometry of the scroll $\PP \End E$. One such generalisation was given in \cite{Hit7} for bundles of Euler characteristic zero, where the tangent cones of 
\[ B^1_{r, r(g-1)} \ = \ \{ E \in U_X (r, r(g-1)) : h^0 (X, E) \ge 1 \} \ \subset \ U_X (r, r(g-1)) \]
another ``generalised theta divisor'', are described geometrically.

In the present work, we generalise the picture in another way. Returning to the opening example, we note that the sequence $0 \to \Ox \to \Ox (D) \to \Ox(D)_D \to 0$ realises $\Ox (D)$ as an \textsl{elementary transformation} of the trivial bundle. Let $V$ be a vector bundle of rank $r$ and $\pi \colon \PP V \to X$ the associated scroll. We consider elementary transformations $0 \to V \to \tV \to \tau \to 0$; that is, bundles $\tV$ of rank $r$ containing $V$ as a locally free subsheaf of full rank. If the support of $\tau$ is reduced and of degree $d$, then the choice of $\tV$ is canonically equivalent to a choice of $d$ points $\nu_1 , \ldots , \nu_d$ of $\PP V$ belonging respectively to fibres over distinct points $x_1 , \ldots x_d$ of $X$. Generalising the definition of $\Ox (D)$, one can realise the sheaf $\tV$ as the sheaf of rational sections of $V$ with poles bounded by the $x_i$ and in the direction corresponding to the $\nu_i$. Then $\tV^*$ is the subsheaf of $V^*$ of sections whose values at $x_i$ belong to the hyperplane determined by $\nu_i$.

Our first goal is to systematise and extend this construction to the case where $\Supp (\tau)$ may be nonreduced. Let $Z$ be a subscheme of $\PP V$ of dimension zero and length $d$. Generalising the operation $D \mapsto \Ox (D)$, we define
\[ V_Z \ := \ \left( \pi_* \left( \cI_Z \otimes \Opvo \right) \right)^* . \]
The association $Z \mapsto V_Z^*$ gives a map $\alpha \colon \Hilb^d (\PP V) \dashrightarrow \Quot^{0, d}(V^*)$, generalising the natural map $\Sym^d X \to \Quot^{0, d}(\Ox)$ given by $D \mapsto \Ox(-D)$. In Theorem \ref{Zexists} and Theorem \ref{surj} we show that the restriction of $\alpha$ to the component of $\Hilb^d (\PP V )$ containing reduced subschemes is surjective and generically injective.

In contrast to the line bundle case, if $r \ge 2$ then $\deg ( V_Z )$ may be strictly less than $\deg V + d$ in special cases. This leads us to the notion of \textsl{$\pi$-nondefectivity} (Definition \ref{defPiNondef}). Much can be said about $\pi$-nondefective subschemes, but we limit ourselves here to what is strictly necessary for the present applications.

In {\S} \ref{grr}, we link the elementary transformations $V_Z$ with the geometry of an image of the scroll $\PP V \dashrightarrow \PP H^1 ( X, V )$. This allows us to prove a generalisation of the geometric Riemann--Roch theorem for scrolls (Theorem \ref{ggrr}). For applications to Brill--Noether loci, it is necessary (and straightforward) to have also a ``relative'' version of this result, given in terms of the geometry of the rank one locus of $V \otimes E$ (Theorem \ref{relggrr}). It should be noted that a similar situation is studied in the recent paper \cite{Br17}; see Remark \ref{Brivio} for discussion. 

In {\S} \ref{TangentConeBN}, we use Theorem \ref{relggrr} together with results in \cite{CT} to give another generalisation of the Riemann--Kempf singularity theorem. Suppose $E$ is a stable, generically generated bundle with $h^0 (X, E ) = r + n$ which is Petri $r$-injective (Definition \ref{PetriRinj}). Then for generic $\Lambda \in \Gr (r, H^0 ( X, E) )$, the bundle $E$ is an elementary transformation of $\Ox^{\oplus r}$. By the results of the previous sections, there exists $Z_\Lambda \in \Hilb^d ( \PP E^* )$ such that $\Ox^{\oplus r} = ( E^* )_{Z_\Lambda}$. This observation leads naturally to a geometric description of the projectivised tangent cone $\TT_E B^r_{r,d}$ in terms of the geometry of the rank one locus $\Delta \subseteq \PP \End E \subset \PP H^1 (X, \End E)$ via the map $\psi \colon \PP \End (E) \dashrightarrow \PP H^1 (X, \End (E))$ mentioned above. The precise statement, which generalises the Riemann--Kempf theorem, is given in Theorem \ref{genRKS}. An interesting corollary is that if $h^0 (X, E) = r+n$, then $\TT_E B^r_{r,d}$ contains the $n$th secant variety of $\Delta$ (Theorem \ref{secant}). This generalises the fact that the tangent cone to the Riemann theta divisor at a point of multiplicity $k+1$ contains the $k$th secant variety of the canonical curve.

We hope that these results may be useful in proving new ``Torelli-type'' statements.

\subsection*{Acknowledgements.} I thank Insong Choe and Michael Hoff for interesting and enjoyable discussions and for many valuable comments on this work.

\subsection*{Notation} We work over an algebraically closed field $\K$ of characteristic zero. If $Y$ is a scheme and $Z \subset Y$ a closed subscheme, we write $\cI_Z$ for the ideal sheaf of $Z$ in $\cO_Y$.

\section{Elementary transformations and finite subschemes of scrolls} \label{eltranssubschemes}

Let $V$ be a vector bundle of rank $r$ over a projective smooth curve $X$, and denote by $\pi \colon \PP V \to X$ the corresponding scroll. The latter is naturally isomorphic to the Quot scheme $\Quot^{0, 1} ( \PP V^* )$ parametrising elementary transformations of the form $0 \to \tV^* \to V^* \to \tau \to 0$, where $\tau$ is a skyscraper sheaf of length 1; equivalently, elementary transformations of the form $0 \to V \to \tV \to \tau \to 0$. More generally, in \cite{Ty} a tower of projective bundles was constructed parametrising such $\tV$ for $\tau$ of degree $d \ge 1$ (see also \cite[Lemma 4.2]{CH1}). Here we study an alternative way of parametrising these elementary transformations. Our approach has some features in common with that of \cite{BGL}, whose \emph{$(r, n)$-divisors} are $\Ox^{\oplus r}$-valued principal parts in the sense below.

Let $Z \subset \PP V$ be a subscheme of dimension zero and length $d$, corresponding to a point of the Hilbert scheme $\Hilb^d ( \PP V )$. Tensoring the sequence $0 \to \cI_Z \to \Opv \to \Oz \to 0$ by $\Opvo$ and taking direct images, we obtain an exact sequence
\begin{equation} 0 \to \pi_* \left(  \cI_Z \otimes \Opvo \right) \to V^* \to \pi_* \left( \Oz \otimes \Opvo \right) \to \cdots \label{defineVZ} \end{equation} 
Set $V_Z := \left( \pi_* \left( \Opvo \otimes \cI_Z \right) \right)^*$. Then $0 \to V \to V_Z \to V_Z / V \to 0$ is an elementary transformation. We write $\tau_Z := V_Z / V$. 

\begin{remark} \label{QuotTerminology} We will switch freely between elementary transformations of the form $V \subset \tV$ and $\tV^* \subset V^*$, depending on what is more convenient at a given time. It will be necessary to distinguish between the statement that $V_Z \cong \tV$ as vector bundles and the stronger statement that $V_Z^* \subset V^*$ and $\tV^* \subset V^*$ define the same point of the Quot scheme $\Quot^{0, d}(V^*)$. To this end, we will occasionally abuse language and speak of an elementary transformation of the form $0 \to V \to \tV \to \tau \to 0$ as ``an element of $\Quot^{0, d} (V^*)$''. \end{remark}

If $r = 1$ then $Z$ is a divisor of degree $d$ on $\PP V = X$, and then $V_Z^* = V^* \otimes \Ox (-Z)$ as points of $\Quot^{0, d} ( V^* )$. 
 For $r \ge 2$ we only have the inequality $\deg ( V_Z ) \le \deg ( V ) + d$, which may be strict. This will be discussed further in {\S} \ref{PiNondefectivity}. For now, let us give the main result of the present section.

\begin{theorem} \label{Zexists} Let $0 \to V \to \tV \to \tau \to 0$ be an elementary transformation where $\tau$ has degree $d \geq 1$.
\begin{enumerate}
\item[(a)] There exists $Z \in \Hilb^d ( \PP V )$ such that $\tV^* = V_Z^*$ as elements of $\Quot^{0, d}(V^*)$. 
\item[(b)] If $\tau$ has reduced support on $X$, then $Z$ is reduced and uniquely determined.
\end{enumerate} \end{theorem}

The proof of this theorem and its refinement Theorem \ref{surj} will occupy the remainder of the section. As there are several ingredients, let us first give an overview. Firstly, note that (b) is almost obvious: In this case, $\tV$ is determined by specifying a line $\Ker \left( V|_x \to \tV|_x \right)$ for each $x \in \Supp ( \tau )$, so we obtain naturally $d$ points of $\PP V$.

If $\tau$ has nonreduced support, then it emerges that a certain choice of $\K$-basis of $H^0 ( X, \tau )$ determines a scheme $Z$ such that $\tV = V_Z$ as elements of $\Quot^{0, d} ( V^* )$. 
 As in the proof of \cite[Theorem 3.1]{Hit7}, we view sections of $V^* \to X$ as sections of $\Opvo \to \PP V$, and show how certain linear conditions defined by the chosen basis elements of $H^0 (X, \tau )$ determine a suitable subscheme $Z$. 

Now to the details. Firstly, let us describe $\tV^*$ more explicitly. We recall that any locally free sheaf $V$ on $X$ has the flasque resolution
\[ 0 \to V \to \sRat(V) \to \sPrin(V) \to 0, \] 
where $\sRat(V)$ is the sheaf of rational sections of $V$, and $\sPrin(V) = \sRat(V)/V$ the sheaf of $V$-valued principal parts\footnote{In the literature there is an ambiguity in terminology: The \textsl{jet sheaf} parametrising germs of local sections of $V$ to order $k$ is sometimes called the ``sheaf of $k$th order principal parts of $V$''. This is a different object from our $\sPrin(V)$.}. The natural $\Ox$-bilinear pairing
\[ \sRat(V) \times \sRat(V^*) \to \sRat(\Ox) \]
induces a well-defined $\Ox$-linear map
\[ \langle \cdot , \cdot \rangle \colon \sPrin(V) \times V^* \to \sPrin(\Ox). \]
The map $f \mapsto \langle \cdot , f \rangle$ in turn defines an $\Ox$-module homomorphism
\[ V^* \to \Hom_{\Ox} \left( \sPrin (V), \sPrin(\Ox) \right). \]
Restricting to the subsheaf $\tau = \tV / V \subset \sPrin(V) = \sRat(V) / V$, we obtain a map
\begin{equation} V^* \to \Hom_{\Ox}( \tau, \sPrin( \Ox )). \label{pptmap} \end{equation}

\begin{proposition} \label{vdual} The subsheaf $\tV^*$ of $V^*$ is the kernel of (\ref{pptmap}). Equivalently,
\[ \tV^* \ = \ \{ f \in V^* : \hbox{$\langle p, f \rangle$ is zero in $\sPrin(\Ox)$ for all $p \in \tau$} \}. \]
\end{proposition}
\begin{proof} 
A section $f$ of $V^*$ defines a regular section of $\tV^*$ if and only if $\langle \tilde{v} , f \rangle$ is a regular function for all $\tilde{v} \in \tV$. Since $f$ is regular, this is equivalent to saying that the principal part $\langle p, f \rangle$ vanishes for all $p \in \tV / V = \tau$. \end{proof}
\begin{proposition} \label{goodframe} Suppose $x \in \Supp ( \tau )$. Let $z$ be a uniformiser at $x$. Then there exist local sections $v_1, \ldots, v_s$ of $V$ near $x$ which are linearly independent at $x$, and positive integers $k_1, \ldots , k_s$ such that $\tau_x$ is spanned over $\Ox$ by the principal parts
\[ \frac{v_1}{z^{k_1}}, \ \ldots, \ \frac{v_s}{z^{k_s}}. \]
In particular, $s \leq r$. \end{proposition}
\begin{proof} 
Let $p_1 , \ldots , p_s$ be a set of generators for $\tau_x$ over ${\Ox}_{,x}$. For $1 \leq j \leq s$, we write $k_j$ for the order of the pole of $p_j$, and reorder so that $k_1 \geq k_2 \geq \cdots \geq k_s$. Then for each $j$, there exists a local section $v_j \in V_x \backslash m_x V_x$ such that $p_j = \frac{v_j}{z^{k_j}}$. (The section $v_j$ is well-defined modulo $m_x^{k_j} V_x$.)
We claim that after reducing $s$ and the $k_j$ if necessary, we may assume that the $v_j$ are linearly independent at $x$.

Suppose there is a nontrivial $\K$-linear combination $\sum_{j=1}^s a_j v_j$ whose value at $x$ is zero; that is, which belongs to $m_x V_x$.
Let $l \in \{1, \ldots, s\}$ be the largest index such that $a_l \ne 0$. Since $k_1 \geq k_2 \geq \cdots \geq k_s$, the function $z^{k_j - k_l}$ belongs to $\cO_{X,x}$ for each $j < l$. Set
\begin{equation} p_l' \ := \ p_l + \sum_{j = 1}^{l-1} z^{k_j - k_l} \frac{a_j}{a_l} p_j. \label{plprime} \end{equation}
A computation shows that
\[ p_l' \ = \ \frac{1}{a_l} \cdot \frac{\left( \sum_{j=1}^l a_j v_j \right)}{z^{k_l}}. \]
This has a pole of order at most $k_l - 1$ at $x$, since the numerator belongs to $m_x V_x$. In light of (\ref{plprime}), the set
\[ p_1, \ \ldots, \ p_{l-1}, \ p_l', \ p_{l+1}, \ \ldots, \ p_s \]
also generates $\tau_x$ over $\Ox$. 
 After performing a finite number of operations of this kind (possibly annihilating some of the $p_j$), we arrive at a generating set $p_1, \ldots, p_s$ of principal parts whose leading coefficients define linearly independent points of $V|_x$. In particular, $s \leq r$.
\end{proof}

If $s < r$, we complete the partial frame $v_1 , \ldots, v_s$ at $x$ to a frame $v_1 , \ldots , v_r$. Write $f_1 , \ldots , f_r$ for the dual frame of $V^*$ near $x$. Note that if $k_1 \ge 2$, these frames carry infinitesimal information also.

\begin{corollary} \label{descrVtildedual} The locally free sheaf $\tV^*$ is spanned near $x$ by
\[ z^{k_1} f_1 , \ z^{k_2} f_2 , \ \ldots, \ z^{k_s} f_s , \ f_{s+1}, \ \ldots , \ f_r . \]
In particular, $\deg \tau_x = k_1 + \cdots + k_s$.
\end{corollary}
\begin{proof} This follows from Propositions \ref{vdual} and \ref{goodframe}. \end{proof}

Each $v_j$ determines uniquely a nonzero point of $V|_x$, given by the image of $v_j$ in the fibre $V_x / m_x V_x$. Hence $v_j$ also determines a point of $\PP V|_x$, which we denote by $\nu_j$.

\begin{remark} \label{loun} If $\length (\tau_x) = 1$ then the generator $p_1 \in V(x)/V$ is unique up to scalar multiple, and the point $\nu_1$ is uniquely determined. However, if $\length (\tau_x) \geq 2$ then the choice of generators $p_1, \ldots , p_s$, and hence the points $\nu_j$ are in general not unique. For example, consider the elementary transformation $0 \to V \to V(x) \to V(x)/V \to 0$. Any frame $\{ v_i \}$ for $V$ near $x$ determines a generating set $\{ \frac{v_i}{z} \}$ for $\tau$ and points $\nu_1 , \ldots , \nu_r$ which span $\PP V|_x$. Furthermore, a general choice of two distinct frames will define different spanning sets for $\PP V|_x$. Note also that the number $s$ of generators of $\tau_x$ over $\Ox$ (as distinct from over $\K$) is not determined by $\length (\tau_x)$. For example, if $\tau_x \cong \cO_x^{\oplus 2}$ then two generators are required over $\Ox$, whereas if $\tau_x \cong \cO_{2x}$ then one suffices. \end{remark}

\begin{proof}[Proof of Theorem \ref{Zexists}]

Firstly, suppose that $\tau$ is supported at a single point $x \in X$. For $1 \le j \le s$, define an elementary transformation $W_j$ of $V^*$ by
\[ W_j = \Ker \left( \frac{v_j}{z^{k_j}} \colon V^* \to \sPrin(\Ox) \right) \]
where the $v_j$ are as in Proposition \ref{goodframe}. 
 Clearly $V^* / W_j$ is supported at $x$, and $W_j$ is spanned near $x$ by
\begin{equation} f_1 , \ \ldots, \ f_{j-1} , \ z^{k_j} f_j , \ \ldots, \ f_s, \ f_{s+1}, \ \ldots , \ f_r , \label{wj} \end{equation}
where the $f_i$ are as defined before Corollary \ref{descrVtildedual}. (Note that $\tV^* = \cap_{j=1}^s W_j$.)

Since $\pi$ is flat, the functor $\pi^*$ is exact. We have a commutative diagram with exact rows:
\begin{equation} \xymatrix{ 0 \ar[r] & \pi^* W_j \ar[d] \ar[r] & \pi^* V^* \ar[d] \ar[r] & \pi^* ( V^* / W_j ) \ar[r] \ar[d] & 0 \\
 0 \ar[r] & \cF_j \ar[r] & \Opvo \ar[r] & \cT_j \ar[r] & 0 } \label{fj} \end{equation}
where $\cF_j$ is the image of the composed map $\pi^* W_j \to \pi^* V^* \to \Opvo$, and $\cT_j$ is the quotient. We will sometimes abuse notation and denote both $\pi^* z$ and its image in (\ref{fj}) simply by $z$; and similarly for $f_1 , \ldots , f_r$. From (\ref{wj}), it follows that $\cF_j \to \Opvo$ is an isomorphism away from $\nu_j$. Furthermore, 
 a $\K$-basis for $\cT_j$ is given by the images of
\[ f_j, \ z f_j, \ \ldots , \ z^{k_j - 1} f_j. \]
In particular, $\length \cT_j = k_j$.

Tensoring the lower row of (\ref{fj}) with the invertible sheaf $\Opv(-1)$, we obtain
\[ 0 \to \cI_j \to \Opv \to \cT_j \otimes \Opv(-1) \to 0 \]
where $\cI_j := \cF_i \otimes \Opv(-1)$ is an ideal sheaf. We write $Z_j$ for the subscheme of $\PP V$ defined by $\cI_j$. Then $\cT_j \otimes \Opv(-1)$ is naturally identified with $\cO_{Z_j}$, whence we see that $Z_j$ has length $k_j$.

We now define
\[ \cI \ := \ \bigcap_{j=1}^s \cI_j. \]
As the support of each $Z_j$ is the isolated point $\nu_j$, we see that $\cI$ is the ideal sheaf of $Z_1 \cup \cdots \cup Z_s =: Z$, a subscheme supported along $\{ \nu_1 , \ldots , \nu_s \}$ of $\PP V$. This $Z$ has dimension zero and length $k_1 + \cdots + k_s$. 
For each $j$, the stalk of $\cI$ at $\nu_j$ coincides with that of $\cI_j$, and is generated by
\begin{equation} z^{k_j}, \frac{f_i}{f_j} \ : \ 1 \leq i \leq r; \ i \neq j. \label{Idescr} \end{equation}
Correspondingly, the stalk of $\cI \otimes \Opvo$ at $\nu_j$ is generated by the images of
\begin{equation} f_1, \ldots, f_{j-1}, z^{k_j}f_j, f_{j+1} , \ldots , f_r . \label{Ioodescr} \end{equation}

We now show that $\pi_* \left( \cI \otimes \Opvo \right) = \widetilde{V}^*$. This is essentially book-keeping. 
Clearly the two sheaves are equal away from $x$. Let $U \subseteq X$ be a neighbourhood of $x$ over which $V$ is trivial. We have
\[ \Opv \left( \pi^{-1} U \right) \ \cong \ \cO_{X, U} \otimes \Gamma \left( \PP^r, \cO_{\PP^r} \right) \ \cong \ \cO_{X, U}. \]
Thus a section $t$ of $\Opvo$ over $\pi^{-1} U$ is of the form $h_1 \cdot f_1 + \cdots + h_r \cdot f_r$
where each $h_j \in \cO_{X, U}$ and the $f_j$ are as above. 
 By (\ref{Ioodescr}), such a $t$ belongs to $\Gamma ( \pi^{-1} U, \cI \otimes \Opvo )$ if and only if
 $h_j \in z^{k_j} \cdot \cO_{X, U}$ for each $j$. By Corollary \ref{descrVtildedual}, this is equivalent to the statement that $t$, viewed as a section of $V^* \to X$, takes its values in $\tV^*$. Thus $\pi_* \left( \cI \otimes \Opvo \right) = \tV^*$. This proves (a) in case $\tau$ is supported at a single point $x$.\\
\par
More generally, if $\tau$ is supported at two or more points, the construction above yields an ideal sheaf $\cI(x)$ and a zero-dimensional subscheme $Z(x)$ of length $d_x$ for each $x \in \Supp ( \tau )$. We write
\[ Z \ := \bigcup_{x \in \Supp (\tau)} Z(x) \quad \hbox{and} \quad \cI_Z \ := \bigcap_{x \in \Supp \tau} \cI(x) .\]
By the local argument above applied to each $x \in \Supp ( \tau )$, we have $\pi_* \left( \cI_Z \otimes \Opvo \right) = \tV^*$. This completes the proof of (a) in general.

(b) If $\tau$ has reduced support, then $d_x = 1$ for each $x \in \Supp ( \tau )$. By Remark \ref{loun}, the intersection of $Z$ with the fibre over $x$ consists of a single, uniquely determined, reduced point. The statement follows. \end{proof}

\subsection*{Hilbert schemes of points} Let us now give a refinement of Theorem \ref{Zexists}. The scheme $\Hilb^d ( \PP V )$ is connected since $\PP V$ is, but by \cite{CCEV} it is not irreducible for $r \ge 4$ and $e \ge 8$. We denote by $\Hilb^d ( \PP V )_0$ the irreducible component containing reduced subschemes.

\begin{theorem} \label{surj} \quad \begin{enumerate}
\item[(a)] The association $Z \ \mapsto \ \pi_* \left( \Opvo \otimes \cI_Z \right)$ defines a rational map
\[ \alpha \colon \Hilb^d ( \PP V ) \ \dashrightarrow \ \Quot^{0, d} (V^*) \]
whose restriction to $\Hilb^d ( \PP V )_0$ is surjective.
\item[(b)] The restriction of $\alpha$ to the subset
\[ \{ Z \in \Hilb^d ( \PP V )_0 : \pi(Z) \hbox{ is reduced}\} \ \subset \ \Sym^d \PP V \backslash \Delta \]
is bijective. In particular, the restriction of $\alpha$ to $\Hilb^d ( \PP V )_0$ is a birational equivalence.
\item[(c)] No other component of $\Hilb^d ( \PP V )$ dominates $\Quot^{0, d} ( V^* )$. \end{enumerate} \end{theorem}

\begin{proof} (a) It is straightforward to globalise the construction $V_Z^* = \pi_* \left( \cI_Z \otimes \Opvo \right)$ and obtain a family of elementary transformations of $V^*$ of degree $-\deg V - d$, parametrised by the locus
\[ \{ Z \in \Hilb^d (\PP V) : \deg ( V_Z ) = \deg V + d \} . \]
As the rank and degree are constant and $X$ has dimension $1$, this family is flat. 
 The existence of $\alpha$ then follows from the universal property of $\Quot$ schemes.

To conclude, by Theorem \ref{Zexists} it will suffice to show that if $Z = Z_1 \cup \cdots \cup Z_s$ is a nonreduced scheme of length $d \ge 2$ arising from the construction in the proof of Theorem \ref{Zexists}, then $Z$ is smoothable. Clearly $Z$ is smoothable if and only if each $Z_j$ is smoothable. Thus it will suffice to prove the smoothability of a $Z$ supported at a single point $\nu_1 \in \PP V|_x$ with length $k \ge 2$. 
 In this case, with the setup of Theorem \ref{Zexists}, we have
\[ \cI_Z \ = \ \left( \pi^* z^k , \frac{\pi^*f_2}{\pi^*f_1} , \ldots , \frac{\pi^*f_r}{\pi^*f_1} \right) \]
where $f_1, \ldots , f_r$ is a frame for $V^*$ on an open subset $U \subseteq X$, and $\Oz$ is generated by the images of $1, \pi^*z, \ldots , \pi^*z^{k-1}$.

Now since $\pi_* \Oz$ is supported at $x$ and generated by $1, z, \ldots , z^{k-1}$, clearly the map $\Ox \to \pi_* \Oz$ is surjective. Therefore, $\pi \colon \PP V \to X$ restricts to a closed embedding $Z \hookrightarrow X$. Thus $Z$ is curvilinear and hence smoothable.

Part (b) follows from Theorem \ref{Zexists} (b). For the rest; clearly $V_Z / V$ has reduced support only if $Z \in \Hilb^d ( \PP V )_0$. As quotients with reduced support are dense in $\Quot^{0, d} ( V^* )$, we obtain (c). \end{proof}

\noindent In the next section, we will describe the indeterminacy locus of $\alpha$ in more detail.

\section{Defective secants} \label{DefectiveSecants}

Here we recall some facts about defective secants, which will be used in several contexts. Let $Y$ be a variety equipped with a line bundle $\cL$ and a map $\psi \colon Y \dashrightarrow |\cL|^*$ (not necessarily base point free). If $Z \subseteq Y$ is a subscheme, then $\Span \left( \psi(Z) \right)$ is the projective linear subspace
\begin{equation} \PP \Ker \left( H^0 (Y, \cL )^* \to H^0 \left( Y, \cI_Z \otimes \cL \right)^* \right). \label{defspan} \end{equation}
When the map $\psi$ is clear from the context, we will simply write $\Span (Z)$. If $Z$ is of dimension zero and reduced, then $\Span (Z)$ is the secant spanned by the images of the points of $Z$. In general, $\Span (Z)$ is a subspace of the span of the union of certain osculating spaces to $\psi (Y)$ at points of $\Supp (Z)$.

\begin{definition} Suppose $Z \subset Y$ has dimension zero. We recall that the \textsl{defect} of $\psi(Z)$ is defined by
\[ \defe \left( \psi(Z) \right) \ := \ \length Z - 1 - \dim ( \Span (\psi(Z) ) ). \]
Again, if the context is clear, we will simply write $\defe (Z)$. We say that $Z$ is \textsl{nondefective} if $\defe (Z) = 0$, and \textsl{defective} otherwise. If $\Span (Z)$ is empty, we define $\dim ( \Span (Z)) = -1$. \end{definition}

\begin{remark} The map $\psi$ is base point free if and only if all $y \in Y$ are nondefective, and an embedding if and only if all $Z \in \Hilb^2 (Y)$ are nondefective. \end{remark}

\begin{remark} Recall that the $n$th secant variety $\Sec^k (Y)$ of a nondegenerate variety $Y \subset \PP^N$ is the Zariski closure of the union of the linear spans of all subsets of $k$ points of $Y$. In general, $Y$ is said to be ``secant defective'' if some $\Sec^k Y$ has less than the expected dimension. The above definition of a defective scheme of dimension zero is a special case of this. 
\end{remark}

Note that if $Z$ has dimension zero, then nondefectivity of $Z$ is equivalent to the surjectivity of the restriction map in (\ref{defspan}). We will use this observation to study the indeterminacy locus of the map $\alpha \colon \Hilb^d ( \PP V ) \dashrightarrow \Quot^{0, d} ( V^* )$ defined in the previous section.

\subsection{Relatively nondefective subschemes} \label{PiNondefectivity}

We return to the situation of Theorem \ref{surj}. The map $\alpha \colon \Hilb^d ( \PP V ) \dashrightarrow \Quot^{0, d} (V^* )$ is defined at $Z$ if and only if $\deg V_Z = \deg V + d$. This is clearly the case for generic $Z \in \Hilb^d ( \PP V)_0$, in particular if $\Supp (\pi(Z))$ consists of $d$ distinct points of $X$.  However, it is not true if for example $Z$ is a union of $r+1$ points $\nu_1 , \ldots , \nu_{r+1}$ in general position in a fibre $\PP V|_x$. Here $V_Z = V \otimes \Ox (x)$ has degree $\deg V + r < \deg V + \length Z$, as the evaluation map
\[ \pi_* ( \Opvo )|_x \ = \ V^*|_x \ \to \ \pi_* ( \Opvo \otimes \Oz ) \ = \ \bigoplus_{i = 1}^{r+1} \nu_i^* \]
is not surjective. This motivates a definition.

\begin{definition} \label{defPiNondef} Let $Z \subset \PP V$ be a subscheme of dimension zero. Recall that we have defined $\tau_Z := V_Z / V$. The subscheme $Z$ will be called \textsl{$\pi$-nondefective} if the following equivalent conditions obtain:
\begin{itemize}
\item the map $V^* \to \pi_* ( \cI_Z \otimes \Opvo )$ in (\ref{defineVZ}) is surjective;
\item the map $V^* / V_Z^* \to \pi_* ( \Oz \otimes \Opvo )$ is an isomorphism;
\item $\deg V_Z = \deg V + \length (Z)$; equivalently $\deg (\tau_Z) = \length (Z)$.
\end{itemize}
Otherwise, $Z$ will be called \textsl{$\pi$-defective}. \end{definition}

Geometrically speaking, $Z$ is $\pi$-defective if and only if for some $x \in X$ the image of $Z \cap \PP V|_x \to \PP^{r-1}$ defined by the restriction of $\Opvo$ is defective. This means that $Z$ is secant defective in $|\Opvo \otimes \pi^* L|^*$ for \emph{any} $L \in \Pic (X)$.

\begin{remark} \begin{enumerate}
\item[(a)] The scheme $Z$ constructed in the proof of Theorem \ref{Zexists} is clearly $\pi$-nondefective.
\item[(b)] As $R^1 \pi_* \Opvo = 0$, a subscheme $Z$ is $\pi$-nondefective if and only if
\[ R^1 \pi_* ( \cI_Z \otimes \Opvo) \ = \ 0 . \]
\item[(c)] By definition, the locus of \emph{$\pi$-defective} subschemes in $\Hilb^d ( \PP V )$ is exactly the indeterminacy locus of $\alpha \colon \Hilb^d ( \PP V ) \dashrightarrow \Quot^{0, d}( V^* )$.
\item[(d)] By either (b) or (c) and the universal property of the Hilbert scheme, it follows that $\pi$-nondefectivity is an open property in (flat) families of zero-dimensional subschemes of length $d$ of $\PP V$. \end{enumerate} \end{remark}

\section{Geometric Riemann--Roch for scrolls} \label{grr}

Let $V \to X$ be any vector bundle with $h^1 (X, V) \ge 1$, and $\pi \colon \PP V \to X$ the corresponding scroll. We begin by describing a map $\PP V \dashrightarrow \PP H^1 (X, V)$. This is a slight generalisation of a construction in \cite[{\S} 3]{HR}, also used in various guises in \cite{CH1}, \cite{Hit7}, \cite{Br17} and elsewhere.

By Serre duality and the projection formula, there are identifications
\begin{equation} H^1 (X, V) \ \cong \ H^0 (X, \Kx \otimes V^* )^* \ = \ H^0 ( \PP V, \pi^* \Kx \otimes \Opv(1) )^* \label{stringofidens} \end{equation}
By standard algebraic geometry, we obtain a map $\psi \colon \PP V \dashrightarrow \PP H^1 (X, V)$. By the proof of \cite[Theorem 3.1]{HR}, generalising a standard fact on line bundles, we have:

\begin{proposition} The map $\psi$ is an embedding if and only if for all effective degree two divisors $x+y$ on $X$ we have
\[ h^0 (X, \Kx(-x-y) \otimes V^*) \ = \ h^0 (X, \Kx \otimes V^*) - 2r; \]
equivalently, if $h^0 (X, V (x+y)) = h^0 (X, V)$ for all $x+y$. \end{proposition}

\noindent The following key result describes a link between the geometry of $\psi ( \PP V )$ and the elementary transformations $V_Z$.

\begin{proposition} \label{SpanZIdentification} Let $Z$ be a subscheme of $\PP V$ (not necessarily $\pi$-nondefective). Then $\Span (Z)$ is the projectivisation of the image of the coboundary map $\partial_Z$ of the sequence
\begin{equation} 0 \to H^0 (X, V) \to H^0 ( X, V_Z ) \to H^0 (X, \tau_Z ) \xrightarrow{\partial_Z} H^1 ( X, V) \to H^1 (X, V_Z ) \to 0 . \label{cohomVZ} \end{equation}
\end{proposition}

\begin{proof} By Serre duality, the map $H^1 ( X, V ) \to H^1 (X, V_Z)$ is identified with
\[ H^0 ( X, \Kx \otimes V^* )^* \ \to \ H^0 (X, \Kx \otimes V_Z^* )^* . \]
By the projection formula and since $V_Z^* = \pi_* ( \cI_Z \otimes \Opvo )$, this becomes in turn
\[ H^0 ( \PP V , \Opvo \otimes \pi^* \Kx )^* \ \to \ H^0 ( \PP V, \cI_Z \otimes \Opvo \otimes \pi^* \Kx )^* . \]
The statement now follows from (\ref{defspan}). \end{proof}

\noindent We give next a generalisation of the geometric Riemann--Roch theorem.

\begin{theorem}[Geometric Riemann--Roch for scrolls] \label{ggrr} Let $0 \to V \to \tV \to \tau \to 0$ be an elementary transformation. Then for any $\pi$-nondefective $Z$ such that $V_Z \cong \tV$ as vector bundles, we have $h^0 (X, \tV ) - h^0 (X, V) = \defe (Z)$. \end{theorem}

\noindent Note that by Theorem \ref{Zexists}, such a $Z$ always exists.

\begin{proof}
Let $Z$ be a zero-dimensional $\pi$-nondefective subscheme of $\PP V$ such that the vector bundles $\tV$ and $V_Z$ are isomorphic. We have:
\begin{align*} \dim (\Span(Z)) \ &= \ \dim (\Image (\partial_Z) ) - 1 \hbox{ by Proposition \ref{SpanZIdentification}} \\
 &= \ h^0 ( X, \tau_Z ) - ( h^0 ( X, \tV ) - h^0 (X, V) ) - 1 \hbox{ by exactness and since $V_Z \cong \tV$} \\
 &= \ \length (Z) - ( h^0 ( X, \tV ) - h^0 (X, V) ) - 1 \hbox{ by $\pi$-nondefectivity.} \end{align*}
Therefore, $( \length (Z) - 1) - \dim ( \Span(Z) ) = h^0 ( X, \tV ) - h^0 (X, V)$. As the expression on the left is exactly $\defe (Z)$, the statement follows. \end{proof}

Note that $Z$ is generally not unique; for example, distinct linearly equivalent effective divisors define isomorphic line bundles. If $V$ is a line bundle, the condition that $V_Z^* = \tV^*$ in $\Quot^{0, d} (V^*)$ uniquely determines $Z$. This is no longer true for $r \ge 2$ (Remark \ref{loun}). We can however give a necessary geometric condition for the equality $V_{Z'}^* = V_Z^*$ in $\Quot^{0, d} (V^*)$.

\begin{proposition} Suppose $Z$ and $Z'$ are such that $V_Z^*$ and $V_{Z'}^*$ define the same point of $\Quot^{0, d} (V^*)$. Then $\Span (Z) = \Span (Z')$. \label{SameSpan} \end{proposition}

\begin{proof} By hypothesis and by definition of the Quot scheme, $V_Z^* = V_{Z'}^*$ as subsheaves of $V^*$. A diagram chasing argument shows that
\begin{equation} \Ker \left( H^1 (X, V) \to H^1 (X, V_Z ) \right) \ = \ \Ker \left( H^1 (X, V) \to H^1 (X, V_{Z'} ) \right) . \label{spanaux} \end{equation}
But by Proposition \ref{SpanZIdentification} and exactness, for any zero-dimensional $Z \subset \PP V$ we have
\[ \Span (Z) \ = \ \PP \Image \left( \Gamma ( \tau_Z ) \to H^1 (X, V) \right) \ = \ \PP \Ker \left( H^1 (X, V) \to H^1 (X, V_Z ) \right) . \]
Putting this together with (\ref{spanaux}), we see that $\Span (Z) = \Span (Z')$.
\end{proof}

\begin{remark} \begin{enumerate}
\renewcommand{\labelenumi}{(\alph{enumi})}
\item Proposition \ref{SameSpan} holds even if $Z$ and $Z'$ are not $\pi$-nondefective, or if they have different lengths. For example, suppose $V$ has rank $2$ and $Z$ is any finite reduced subscheme of length $d \ge 2$ of a fibre $\PP V|_x$. Then $\Span (Z)$ is the fibre $\PP V|_x$ and $V_Z$ is isomorphic to $V \otimes \Ox (x)$, independently of $\length (Z)$.
\item The converse of the proposition does not hold. For example, if $h^1 ( X, V ) = 1$ then $\psi$ is constant.
\end{enumerate} \end{remark}

\begin{remark} Suppose $V = \Ox$, so $\pi \colon \PP V \to X$ is the identity map, and
\[ \psi \colon \PP \Ox \ = \ X \ \to \ \PP H^1 (X, \Ox) \ = \ | \Kx |^* \]
is the canonical map. Then $\Hilb^d X = \Sym^d X$ parametrises effective divisors of degree $d$ on $X$. Trivially, all such $Z$ are $\pi$-nondefective, and we obtain the usual geometric Riemann--Roch theorem. \end{remark}

\noindent Before making the next remark, and in view of the situation to be studied in {\S} \ref{TangentConeBN}, for completeness we recall the definition of a stable vector bundle.

\begin{definition} A vector bundle $V \to X$ is said to be \textsl{stable} if for each proper subbundle $W \subset V$ we have $\frac{\deg(W)}{\rank(W)} < \frac{\deg(V)}{\rank(V)}$. \end{definition}

\noindent As is well known, stability is an open condition on families of vector bundles of fixed rank and degree.

\begin{remark} (A generalised Abel--Jacobi map) For $d \ge 1$, let $\alpha_d \colon \Sym^d X \to \Pic^d (X)$ be the Abel--Jacobi map $D \mapsto \Ox ( D )$. As was classically known, $\alpha_d^{-1} ( L )$ is exactly the linear series $|L|$. More generally; fix a bundle $V$ of rank $r$ and degree $e-d$. Sending a length $d$ subscheme $Z \subset \PP V$ to the moduli point of $V_Z$ in $\ure$ defines a rational map
\[ \alpha_{V, d} \colon \Hilb^d ( \PP V ) \ \dashrightarrow \ \ure , \]
generalising the Abel--Jacobi map. In particular, if $V = \Ox^{\oplus r}$ then, as in the rank one case, the image of $\alpha_{V, d}$ is exactly the locus of stable, generically generated bundles. Moreover, by Theorem \ref{ggrr}, the preimage of
\[ \{ E \in \urd : E \hbox{ is generically generated and } h^0 (X, E) \ge r+1 \} \ \subseteq B^{r+1}_{r,d} \]
is exactly the locus
\[ \left\{ Z \in \Hilb^d ( \Ox^{\oplus r} ) : \psi ( Z ) \hbox{ is defective in } \PP H^1 ( \Ox^{\oplus r} ) \hbox{ and } \left( \Ox^{\oplus r} \right)_Z \hbox{ is stable} \right\} . \]
For $r \ge 2$, the situation is complicated by the requirement of stability (but see \cite{BBPN2}) and the presence of nontrivial automorphisms of $X \times \PP^r$. Nonetheless, viewing sums of points on $X$ as \emph{zero-dimensional subschemes} (as opposed to \emph{codimension one subschemes}), Theorems \ref{surj} (a) and \ref{ggrr} give a natural generalisation of the picture for line bundles and linear series on $X$ to bundles of higher rank. \end{remark}

\subsection*{A relative version}

For applications to Brill--Noether loci, we will need a more general version of Theorem \ref{ggrr}. Let $V$ and $F$ be vector bundles over $X$. For any $Z \in \Hilb^d ( \PP V )$, we have an exact sequence
\[ 0 \ \to \ V \otimes F \ \to \ V_Z \otimes F \ \to \ \tau_Z \otimes F \ \to \ 0 . \]
Inside the scroll $\PP ( V \otimes F )$ we have the rank one locus $\Delta := \PP V \times_X \PP F$, also called the \textsl{decomposable locus}. There is a commutative diagram
\begin{equation} \xymatrix{ \Delta \ar[r]^\tomega \ar[d]_\tpi & \PP V \ar[d]^\pi \\
 \PP F \ar[r]^\omega & X. } \label{DeltaMaps} \end{equation}
Since $\PP F \to X$ is flat, for any closed subscheme $Z \subset \PP V$ we have $\cI_{Z \times_X \PP F} \ = \ \tomega^* \cI_Z$. Thus if $\cL \to \Delta$ is a line bundle and $\psi \colon \Delta \dashrightarrow | \cL |^*$ the associated map, in view of (\ref{defspan}) we obtain
\begin{equation} \Span ( \psi ( Z \times_X \PP F ) ) \ = \ \PP \Ker \left( H^0 ( \Delta, \cL )^* \ \to \ H^0 ( \Delta , \tomega^* \cI_Z \otimes \cL )^* \right) . \label{reldefspan} \end{equation}
Note that $\Delta$ is just the projective bundle $\PP (\pi^* F) \to \PP V$. If $Z \in \Hilb^d ( \PP V )$ is reduced, $Z \times_X \PP F$ is a union of $d$ fibres of $\PP F$. 
 In general, the expected dimension of $\Span ( Z \times_X \PP F )$ is $\rank (F) \cdot \length (Z) - 1$.

\begin{definition} \label{RelativeDefect} For $V$, $Z$, $F$ and $\psi$ as above, the \textsl{defect} of $\psi( Z \times_X \PP F)$ is
\[ \defe ( Z \times_X \PP F ) \ := \ (\rank ( F ) \cdot \length (Z) - 1 ) - \dim ( \Span ( Z \times_X \PP F) ) . \]
\end{definition}

\noindent Now we can generalise the previous results of this section. We write $\cL$ for the line bundle
\[ (\tpi^* \omega^* \Kx) \otimes ( \tpi^* \Opfo ) \otimes ( \tomega^* \Opvo ) \ \to \ \Delta . \]

\begin{theorem} \label{relggrr} \quad \begin{enumerate}
\item[(a)] There is a natural identification $H^1 ( X, V \otimes F ) \ \isom \ H^0 ( \Delta , \cL )^*$. In particular, there is a natural map $\Delta \dashrightarrow \PP H^1 ( X, V \otimes F )$, which we again denote $\psi$.
\item[(b)] Via the above identification, $\Span ( Z \times_X \PP F )$ coincides with the projectivised image of the coboundary map $\partial_Z$ in the sequence
\begin{multline} 0 \ \to \ H^0 ( X, V \otimes F ) \ \to \ H^0 ( X, V_Z \otimes F ) \ \to \ H^0 ( X, \tau_Z \otimes F ) \\
 \xrightarrow{\partial_Z} \ H^1 ( X, V \otimes F ) \ \to \ H^1 ( X, V_Z \otimes F ) \ \to \ 0 . \label{relSpanZIdentif} \end{multline}
\item[(c)] (Relative generalised geometric Riemann--Roch) For any $\pi$-nondefective $Z$ such that $V_Z \cong \tV$ as vector bundles, we have
\[ h^0 ( X, \tV \otimes F ) - h^0 ( X, V \otimes F ) \ = \ \defe \left( \psi \left( Z \times_X \PP F \right) \right). \]
\item[(d)] Suppose $Z$ and $Z'$ are such that $V_Z^*$ and $V_{Z'}^*$ define the same point of $\Quot^{0, d}(V^*)$. Then $\Span (Z \times_X \PP F ) = \Span (Z' \times_X \PP F)$. \end{enumerate} \end{theorem}

\begin{proof} (a) This is a technical but straightforward computation. By Serre duality and repeated use of the projection formula, we get an identification
\begin{equation} H^1 ( X, V \otimes F ) \ \isom \ H^0 \left( \PP V , \pi^* \left\{ \omega_* \left( \omega^* \Kx \otimes \Opfo \right) \right\} \otimes \Opvo \right)^* . \label{auxident} \end{equation}
As $\pi \colon \PP V \to X$ is flat, by \cite[Proposition III.9.3]{Har} there is a canonical isomorphism
\[ \pi^* ( \omega_* \cS ) \ \isom \ \tomega_* ( \tpi^* \cS ) \]
for any coherent sheaf $\cS$ on $\PP F$. Setting $\cS = \omega^* \Kx \otimes \Opfo$ in (\ref{auxident}), we obtain
\[ H^0 \left( \PP V , \tomega_* \left\{ \tpi^* \left( \omega^* \Kx \otimes \Opfo \right) \right\} \otimes \Opvo \right)^* . \]
Using the projection formula, we see easily that this is identified with
\[ H^0 \left( \Delta , (\tpi^* \omega^* \Kx) \otimes ( \tpi^* \Opfo ) \otimes ( \tomega^* \Opvo ) \right)^*, \]
which is exactly $H^0 ( \Delta, \cL )^*$.

(b) A calculation similar to that in (a) shows that via the identification in (a), the subspace
\[ H^1 ( X, V_Z \otimes F )^* \ \cong \ H^0 ( X, \Kx \otimes F^* \otimes V_Z^* ) \ \subset \ H^0 ( X, \Kx \otimes F^* \otimes V^* ) \]
is identified with $H^0 \left( \Delta , \cL \otimes \tomega^* \cI_Z \right) \subseteq \ H^0 ( \Delta , \cL )$. Then the statement follows from (\ref{reldefspan}).

Parts (c) and (d) can be proven exactly as Theorem \ref{ggrr} and Proposition \ref{SameSpan} respectively. \end{proof}

\begin{remark} \label{Brivio} The idea behind Theorem \ref{relggrr} (b)--(c) is present in a recent work of Brivio \cite{Br17}. Let $E$ be a general bundle in $U_X (r, r(2g-1))$ and $D = x_1 + \cdots + x_g$ an effective divisor of degree $g$. The condition of interest in \cite{Br17} is that $h^0 (X, E(-D)) = 1$. As $\chi (X, E(-D) ) = 0$, this is equivalent to $h^0 ( X, \Kx \otimes E^*(D)) = 1$. These spaces appear in the cohomology sequence
\begin{multline*} 0 \ \to \ H^0 ( X, \Kx \otimes E^* ) \ \to \ H^0 ( X, \Kx \otimes E^*(D) ) \ \to \ H^0 (X, \Kx \otimes E^*(D)|_D ) \\
 \to \ H^0 ( X, E )^* \ \to \ H^0 ( X, E(-D) )^* \ \to \ 0 . \end{multline*}
This is exactly (\ref{relSpanZIdentif}) with $V = \Kx(-D)$ and $Z = D$, and $F = E^*$. Then by Theorem \ref{relggrr} (b), we can interpret the coboundary map as the restriction of the \emph{tautological model} of \cite{Br17} to the fibres of $\PP ( \Kx \otimes E^* ) \cong \PP E^*$ along $D$. (Note that $\PP E^*$ is denoted by $\PP E$ in \cite{Br17}.) Then Theorem \ref{relggrr} (c) gives another proof that $h^0 ( X, E(-D) ) = h^0 (X, \Kx \otimes E^*(D)) = 1$ if and only if the fibres $\PP E^*|_{x_1} , \ldots , \PP E^*|_{x_g}$ span a space of dimension one less than expected; equivalently, that there exist
\[ ( \eta_1 , \ldots \eta_g ) \ \in \ \PP E^*|_{x_1} \times \cdots \times \PP E^*|_{x_g} \]
which are linearly dependent in $\PP H^0 ( X, E)^*$. (Compare with \cite[Lemma 5.1 and Proposition 7.2]{Br17}). \end{remark}

\section{Tangent cones of higher rank Brill--Noether loci} \label{TangentConeBN}

Suppose $L \to X$ is an effective line bundle of degree $d \le g$. Then $L$ defines a point of the Brill--Noether locus
\[ W_d \ = \ \{ L \in \Pic^d (X) : h^0 ( X, L) \ge 1 \} . \]
The projectivised tangent cone $\TT_L W_d$ at $L$ belongs to
\[ \PP T_L \Pic^d (X) \ = \ \PP H^1 ( X, \Ox ) \ = \ |\Kx|^* . \]
The \textsl{Riemann--Kempf singularity theorem} (see \cite[Chapter 2]{GrHa}) states that
\[ \TT_L W_d \ = \ \bigcup_{L \in |D|} \Span \left( \phi_{\Kx} ( D ) \right) \]
where $\phi_{\Kx}$ is the canonical map. We will generalise this picture to bundles of higher rank, using Theorems \ref{surj} and \ref{ggrr}.

\subsection{Higher rank Brill--Noether loci} Here we recall briefly some essential facts, referring the reader to \cite{GT} and \cite{CT} for details. The moduli space $\urd$ of stable bundles of rank $r$ and degree $d$ over $X$ is an irreducible quasi-projective variety of dimension $r^2 (g-1)+1$. The \textsl{Brill--Noether locus} $B^k_{r, d}$ is defined set-theoretically by
\[ B^k_{r,d} \ = \ \{ E \in \urd : h^0 (X, E) \ge k \} . \]
This is a determinantal subvariety of $\urd$, with expected codimension $k (k - d + r(g-1) )$, when this is nonnegative. Suppose $E \in \urd$ satisfies $h^0 ( X, E) = k$. Then the Zariski tangent space $T_E B^k_{r,d}$ is exactly $\Image ( \mu )^\perp$, where $\mu$ is the \textsl{Petri map}
\[ H^0 ( X, E ) \otimes H^0 ( X, \Kx \otimes E^* ) \ \to \ H^0 ( X, \Kx \otimes \End E ) . \]
Equivalently, via Serre duality we have
\[ T_E B^k_{r,d} \ = \ \Ker \left( \cup \colon H^1 ( X, \End E ) \to \Hom \left( H^0 (X, E) , H^1 (X, E) \right) \right) . \]
Thus $B^k_{r,d}$ is smooth and of the expected dimension at $E$ if and only if $\mu$ is injective; equivalently, if $\cup$ is surjective.

More generally, suppose $h^0 ( X, E ) = k \ge m \ge 1$. We denote $\Gr(m, H^0 ( X, E))$ simply by $\Gr$ to ease notation. Write $\cU$ for the universal bundle over $\Gr$, and consider the diagram
\[ \xymatrix{ \cO_{\Gr} \otimes H^1 ( X, \End E ) \ar[r]^-\cup \ar[dr]_-{\tcup} & \cO_{\Gr} \otimes \Hom \left( H^0 ( X, E ) , H^1 ( X, E ) \right) \ar[d] \\
 & \Hom \left( \cU , \cO_{\Gr} \otimes H^1 ( X, E ) \right) } . \]

\begin{definition} \label{PetriRinj} (cf. \cite[{\S} 2]{CT}) Let $E$ be as above. If the restricted Petri map
\[ \mu_\Lambda \ := \ \mu|_{\Lambda \otimes H^0 ( X, \Kx \otimes E^* )} \]
is injective for all subspaces $\Lambda \in \Gr( m , H^0 ( X, E ))$, then $E$ is said to be \textsl{Petri $m$-injective}. Equivalently, $E$ is Petri $m$-injective if and only if
\begin{equation} \tcup|_\Lambda \ \colon H^1 ( X , \End E ) \ \to \ \Hom \left( \Lambda , H^1 ( X, E ) \right) \label{tcupLambda} \end{equation}
is surjective for all $\Lambda \in \Gr$. \end{definition}

If $E$ is Petri $m$-injective, then $\Ker ( \tcup )$ is a vector subbundle of $\cO_{\Gr} \otimes H^1 ( X, \End E)$. By 
 \cite[Theorem 2.4 (4)]{CT}, the projectivised tangent cone $\TT_E B^m_{r, d}$ is given by
\[ \TT_E B^m_{r, d} \ = \ \bigcup_{\Lambda \in \Gr} \Image ( \mu_\Lambda )^\perp \ = \ \bigcup_{\Lambda \in \Gr} \Ker ( \tcup|_\Lambda ) , \]
the second equality following from Serre duality as above. Thus, in summary we obtain:

\begin{lemma} Suppose $E$ is a stable bundle with $h^0 ( X, E ) = k \ge m \ge 1$ which is Petri $m$-injective. Then $\TT_E B^m_{r,d}$ is the image of the scroll
\begin{equation} S \ := \ \PP \Ker ( \tcup ) \ \subseteq \ \Gr \times \PP H^1 ( X, \End E ) \label{S} \end{equation}
by the projection to $\PP H^1 ( X, \End E )$. In particular, the tangent cone is closed and irreducible. \label{TTdescr} \end{lemma}

\noindent Note that as all fibres of $S$ contain the subspace
\[ \PP \Ker \left( \cup \colon H^1 ( X, \End E ) \to \Hom ( H^0 ( X, E), H^1 ( X, E ) \right) \ = \ T_E B^k_{r, d} , \]
the map $S \to \PP H^1 (X, \End E )$ is an embedding only if $\cup$ is injective; equivalently, if $\mu$ is surjective.

\subsection{Riemann--Kempf for generically generated bundles} Throughout this section, $E$ will be a stable bundle with $h^0 ( X, E ) = k \ge r$ and which is generically generated and Petri $r$-injective. Setting $V = E^*$ and $F = E$, we consider the map $\psi \colon \PP \End E \dashrightarrow \PP H^1 ( X, \End E )$ as defined in {\S} \ref{grr}. Although we will not require this fact, we note that by \cite[Theorem 3.1]{HR}, this is an embedding for general $E$ and $X$ if $g \ge 5$.

We now take $m = r$ and write $\Gr := \Gr ( r, H^0 ( X, E) )$. Set
\[ U \ := \ \{ \Lambda \in \Gr : \ev_\Lambda \colon \Ox \otimes \Lambda \to E \hbox{ is generically injective} \} . \]
This is an open subset of $\Gr$, which by hypothesis is dense.
For each $\Lambda \in U$, transposing the evaluation map $\Ox \otimes \Lambda \to E$, we obtain
\[ 0 \to E^* \to \Ox \otimes \Lambda^* \to \tau \to 0 \]
where $\tau$ is a torsion sheaf of degree $d$, supported along a divisor in $|\det E|$. Hence, by Theorem \ref{Zexists}, there exists $Z_\Lambda \in \Hilb^d ( \PP E^* )$ such that the elementary transformation $\Ox \otimes \nolinebreak \Lambda^*$ coincides with $(E^*)_{Z_\Lambda}$ as points of $\Quot^{0, d} ( E )$ (cf{.} Remark \ref{QuotTerminology}). Moreover, there is a commutative diagram
\[ \xymatrix{ & H^1 ( X, \End E) \ar[dl] \ar[dr]^{\tcup|_\Lambda} & \\
H^1 (X, \Lambda^* \otimes E) \ar[rr]^\sim_\cup & & \Hom ( \Lambda , H^1 (X, E)) } , \]
whence $\Ker ( \tcup|_\Lambda ) \ = \ \Ker \left( H^1 ( X, \End E ) \ \to \ H^1 ( X, \Lambda^* \otimes E ) \right)$. 
 Therefore, by exactness and by Theorem \ref{relggrr} (b), with
\[ V \ = \ E^* \quad \hbox{and} \quad V_Z \ = \ (E^*)_{Z_\Lambda} \ = \ \Ox \otimes \Lambda^* \quad \hbox{and} \quad F \ = \ E , \]
we have
\begin{equation} \PP \Ker ( \tcup|_\Lambda ) \ = \ \Span ( Z \times_X \PP E ) . \label{BNspanIdentif} \end{equation}
Note that although $Z_\Lambda$ may not be unique, $\Span ( Z \times_X \PP E )$ is independent of the choice of $Z_\Lambda$ by Theorem \ref{relggrr} (d).

We consider a useful special case. Set
\[ U_1 \ := \ \{ \Lambda \in U : \Supp ( \tau_\Lambda ) \hbox{ is reduced} \} . \]
Clearly this is an open subset of $U$. 
 By an argument using Bertini's theorem, the locus $U_1$ is nonempty for example if $E$ is globally generated.
 If $\Lambda \in U_1$, then $Z_\Lambda$ is reduced and uniquely determined, by Theorem \ref{Zexists} (b). Explicitly, $Z_\Lambda = \{ \nu_1 , \ldots , \nu_d \}$ where $\nu_i = \Ker ( E^*|_{x_i} \to \K^r )$ for distinct points $x_1 , \ldots , x_d \in X$. Equivalently, $\nu_i \in E^*|_{x_i}$ defines the hyperplane $\Image \left( \Lambda \to E|_{x_i} \right)$. In this case,
\begin{equation} Z_{\Lambda} \times_X \PP E \ = \ \bigcup_{i=1}^d \PP ( \nu_i \otimes E|_{x_i} ) \label{UoneZxPF} \end{equation}


\begin{theorem}[Generalised Riemann--Kempf singularity] \label{genRKS} Suppose $E$ is a stable bundle of degree $d < rg$ with $h^0 (X, E) = k \ge r$ which is Petri $r$-injective and generically generated.
\begin{enumerate}
\item[(a)] For any $\Lambda \in U$, we have $\defe ( Z_{\Lambda} \times_X \PP E ) \ = \ kr - 1$.
\item[(b)] The tangent cone to $B^r_{r,d}$ at $E$ is the Zariski closure of
\begin{equation} \bigcup_{\Lambda \in U} \Span \left( Z_\Lambda \times_X \PP E \right) . \label{U} \end{equation}
\item[(c)] Suppose $U_1$ is nonempty (for example, if $E$ is globally generated). Then the tangent cone to $B^r_{r, d}$ at $E$ is the Zariski closure of
\begin{equation} \bigcup_{\Lambda \in U_1} \Span \left( \bigcup_{\nu \in Z_\Lambda} \PP \left( \nu \otimes E|_{\pi(\nu)} \right) \right) . \label{Uone} \end{equation}
\end{enumerate} \end{theorem}

\noindent Note that the hypothesis $d < rg$ ensures that $B^r_{r,d}$ is a proper sublocus of $\urd$.

\begin{proof} (a) By Theorem \ref{relggrr} (c), we have
\[ \defe \left( Z_{\Lambda} \times_X \PP E \right) \ = \ h^0 ( X, \Lambda^* \otimes E ) - h^0 ( X, E^* \otimes E) \ = kr - 1 . \]

(b) By (\ref{BNspanIdentif}), we see that the locus (\ref{U}) is the image of the dense open subset $S|_U$, where $S$ is the scroll defined in (\ref{S}). By Petri $r$-injectivity and Lemma \ref{TTdescr}, the tangent cone is closed and irreducible. Hence by a topological argument it is exactly the closure of (\ref{U}). 

(c) In view of (\ref{UoneZxPF}), the locus (\ref{Uone}) is the image of the dense open subset $S|_{U_1}$. As in part (b), the closure of (\ref{Uone}) is the tangent cone. \end{proof}

\begin{remark} The hypothesis of Petri $r$-injectivity is only required in the proofs of (b) and (c). In general, the loci (\ref{U}) and (\ref{Uone}) are contained in $\TT_E B^r_{r,d}$. \end{remark}

\subsection{Secant varieties} In \cite[{\S} 5]{CT} it is shown that the tangent cones to certain generalised theta divisors contain secant varieties of the curve $X$. Here we deduce a similar statement for $\TT_E B^k_{r,d}$ using Theorem \ref{genRKS}.

\begin{theorem} \label{secant} Suppose $E$ is a generically generated stable bundle of degree $d < rg$ with $h^0 (X, E) = r+n$ for some $n \ge 1$. Then $\TT_E B^r_{r,d}$ contains $\Sec^n \Delta$. \end{theorem}

\begin{proof} Let $\nu_1 \otimes e_1 , \ldots , \nu_n \otimes e_n$ be a collection of $n$ points of $\Delta \subset \PP (E^* \otimes E)$ lying over distinct $x_1 , \ldots , x_n \in X$ respectively. Since $E$ is generically generated, the condition that a section $s$ satisfy $s(x_i) \in \Ker (\nu_i)$ for $1 \le i \le n$ determines a linear subspace of codimension at most $n$ in $H^0 (X, E)$. Thus there exists $\Lambda \in \Gr$ such that $Z_\Lambda$ contains the points $\nu_1 , \ldots , \nu_n$. (If the $\nu_i$ are general enough, $\Lambda$ is unique.) By Theorem \ref{genRKS} (b), the tangent cone contains
\[ \Span \left( Z \times_X \PP E \right) \ = \ \Span \left( \bigcup_{i=1}^n \PP ( \nu_i \otimes E|_{x_i} ) \right) . \]
hence in particular the secant spanned by $\nu_1 \otimes e_1 , \ldots , \nu_n \otimes e_n$. Since the $\nu_i$ were chosen generally, $\TT_E B^r_{r, d}$ contains a dense subset of $\Sec^n ( \Delta )$, and hence all of $\Sec^n ( \Delta )$ since $\TT_E B^r_{r,d}$ is closed. \end{proof}

\subsection{Existence of good singular points} It is nontrivial to establish that there exist bundles $E$ satisfying the hypotheses of Theorem \ref{genRKS}. For $k = r+1$, sufficient conditions on $d$ and $g$ are given in \cite[Proposition 6.6]{BBPN}. 

If we relax the condition of Petri $r$-injectivity and only ask that $E$ be stable and generically generated, then Theorem \ref{secant} is still valid. Such $E$ can be constructed using a method in \cite{Mer}, which we recall for the reader's convenience. Let $L_1 , \ldots , L_r$ be mutually nonisomorphic line bundles of degree $e < d/r$. Consider elementary transformations
\begin{equation} 0 \ \to \ L_1 \oplus \cdots \oplus L_r \ \to \ E \ \to \ \tau \ \to \ 0 \label{E} \end{equation}
where $\tau$ is a torsion sheaf of degree $f \ge 1$. By \cite[Th\'eor\`eme A-5]{Mer}, a general such $E$ is stable for any choice of $L_1, \ldots , L_r$.

We require that $E$ also be generically generated and satisfy $h^0 ( X, E ) = r + n$ for some $n \ge 1$. As above, we assume $d < rg$. Taking $L_i$ to be effective of degree $e < d/r$ and $\tau$ of degree $f = d - re$, we obtain a generically generated stable $E$ of degree $d$. If $h^0 (X, L_i) \ge 2$ for at least one $i$, then $E$ is a singular point of $B^r_{r,d}$ and satisfies the hypotheses of Theorem \ref{secant}. Note that in this case $d \ge r \cdot \mathrm{gon}(X) + 1$.

\begin{remark} In fact, it is possible to show the following. Let $X$ be a general curve of genus $g$. Suppose the following are satisfied:
\begin{enumerate}
\item[(i)] $1 \le e \le g-1$
\item[(ii)] $g - m(m - e + g - 1 ) \ge 1$
\item[(iii)] $d := re + f < rg$
\end{enumerate}
Then for $r \le k \le rm$, there exist stable, generically generated $E$ of rank $r$ and degree $d$ with $h^0 ( X, E ) = k$ and $\mu \colon H^0 ( X, E ) \otimes H^0 ( X, \Kx \otimes E^* ) \to H^0 ( X, \Kx \otimes \End E )$ injective.

\begin{proof}[Idea of proof] In view of (ii), by the Brill--Noether theory of line bundles, we may choose mutually nonisomorphic $L_1 , \ldots , L_r$ of degree $e$ with $h^0 (X, L_i) = m$. Using (i) and (iii), the construction (\ref{E}) yields a stable, generically generated $E$ of degree $d$ with $h^0 ( X, E ) = rm$. Then an argument similar to that in \cite[Th\'eor\`eme 1.2]{Hir2}, using the generality of $X$, shows that a general such $E$ is in fact Petri injective. The theory of $\chi$-like stratifications \cite{Hir1} then shows that $E$ can be deformed to a stable, generically generated, Petri injective $E'$ with $h^0 ( X, E ) = k$ for $r \le k \le rm - 1$. \end{proof}

\noindent As a more general statement will be proven in a forthcoming paper, we omit the details. \end{remark}

\bibliography{grr_tangentconeBN}{}

\bibliographystyle{alpha}

\end{document}